\documentclass[a4paper,12pt]{amsart} 

\usepackage{amssymb,amsmath}
\usepackage{amscd}
\usepackage{amsthm}
\usepackage{mathrsfs}

\newtheorem{definition}{Definition}[section]
\newtheorem{theorem}{Theorem}[section]

\newtheorem{proposition}{Proposition}[section]
\newtheorem{cor}{Corollary}[section]

\def\bbr{\mathbb{R}}

\def\bbz{\mathbb{Z}}

\def\scr#1{\mathscr{#1}}

\def\cf{\mathscr{F}}
\def\ch{\mathscr{H}}
\def\ci{{I}}
\def\cj{{J}}
\def\cp{\mathscr{P}}

\def\cm{\mathscr{M}}
\def\cn{\mathscr{N}}

\def\<{{\langle}}
\def\>{\rangle}

\def\iff{\longleftrightarrow}
\def\then{\longrightarrow}

\title{Group actions on Polish spaces}
\author{Robert Ra{\l}owski and Szymon {\.Z}eberski}\thanks{The work has been partially financed by NCN means granted by decision DEC-2011/01/B/ST1/01439 and by grant { 8211204601, 8211104160, MPK: 9120730000} from the Faculty of Pure and Applied Mathematics, Wrocław University of Science and Technology.}
\address{ Faculty of Pure and Applied Mathematics,
         Wroc{\l}aw University of Science and Technology, Wybrze\.ze Wyspia\'n\-skie\-go 27, 
         50-370 Wroc{\l}aw, Poland.}

\email[Robert Ra{\l}owski]{robert.ralowski@pwr.edu.pl}
\email[Szymon \.Zeberski]{szymon.zeberski@pwr.edu.pl}


\date{}

\begin{document}

\hyphenation{non-mea-su-ra-ble}

\maketitle

\begin{abstract} In this paper we investigate the action of Polish groups (not necessary abelian) on an uncountable Polish spaces. We consider two main situations. First, when the orbits given by group action are small and the second when the family of orbits are at most countable. We have found  some subgroups which are not measurable with respect to a given $\sigma$-ideals on the group and the action on  some subsets gives a completely nonmeasurable sets with respect to some $\sigma$-ideals with a Borel base on the Polish space.  In most cases the general results are consistent with ZFC theory and are strictly connected with cardinal coefficients. We give  some suitable examples, namely the subgroup of isometries of the Cantor space where the orbits are sufficiently small. In the opposite case we give an example of the group of the homeomorphisms of a Polish space in which there is a large orbit and we have found the subgroup without  Baire property and a subset of the mentioned space such that the action of this subgroup on this set is completely nonmeasurable set with respect to the $\sigma$-ideal of the subsets of first category.
\end{abstract}

\section{Notation and Terminology}

In this paper we will use standard set theoretic notation following \cite{Jech}. In particular, for any set $X$ and any cardinal $\kappa$ we denote by $[X]^{<\kappa}$  the family of all subsets of $X$ with size less than $\kappa$. Moreover, $\scr{P}(X)$ denotes the power set of $X$.

Let $X$ be any uncountable Polish space. By $\scr{B}(X)$ we denote the $\sigma$-algebra of all Borel subsets of $X$.

$I\subseteq \cp(X)$ is an ideal on $X$ if $I$ is closed under taking subsets and under finite unions. We assume that $I$ is nontrivial meaning that $I\neq \emptyset$ and $X\notin I.$ If for every sequence $(A_n)_{n\in\omega}$ of elements of $I$ we have that $\bigcup_{n\in\omega}A_n\in I$ then $I$ is called a $\sigma$-ideal.

Let us recall that $\scr{F}\subset I$ is a base of $I$ iff 
$$
(\forall Y\in I)(\exists B\in\scr{F})\;\; Y\subset B.
$$
If, additionally, $\scr{F}$ consists of Borel sets then we say that $\scr{F}$ is a Borel base of $I.$

We will use the following cardinal coefficients:
\begin{definition}Let $I\subset \scr{P}(X)$ be a $\sigma$-ideal on a Polish space $X$. Assume that $I$ has a Borel base. Set
$$
cov(I)=min\{ |\scr{A}|:\; \scr{A}\subset I\land \bigcup \scr{A} = X\},
$$
$$
cov_h(I)=min\{ |\scr{A}|:\; (\scr{A}\subset I)\land (\exists B\in \scr{B}(X)\setminus I)\;(\bigcup \scr{A}= B)\},
$$
$$
cof(I)=min\{ |\scr{A}|:\; \scr{A}\subset I\land \scr{A}\text{ is a  Borel base of } I\},
$$
$$
non(I) = min\{ |A|:\; A\subseteq X \land  A\notin I \}.
$$
\end{definition}

\begin{definition}\label{pis} We say that a pair $(X,I)$ is a Polish ideal space iff
\begin{itemize}
 \item $X$ is an uncountable Polish space,
 \item $I\subset\scr{P}(X)$ is a $\sigma$-ideal containing singletons and having a Borel base.
\end{itemize}

\end{definition}

Let us remark that a classical example of a Polish ideal space is $(X,\cm)$, where $X$ is an uncountable Polish space without isolated points and $\cm$ stands for the $\sigma$-ideal of meager sets, i.e. $\sigma$-ideal generated by closed nowhere dense sets. 

Another classical example is $\sigma$-ideal $\cn$ of the null subsets of $X,$ where $X$ is $\bbr, [0,1], 2^\omega.$ If $X$ is $\bbr$ or $[0,1]$ we consider the Lebesgue measure and if $X$ is $2^\omega$ we consider the Haar measure.

The cardinal coefficients connected to $\cm$ and $\cn$ do not depend on $X$ (see e.g. \cite{Kechris}). Moreover $cov(\cm)=cov_h(\cm),$ $cov(\cn)=cov_h(\cn).$

By $\scr{B}^+_I(X)$ we denote the set $\scr{B}(X)\setminus I$ of all $I$-positive Borel subsets of the space $X$.
Let us remark that if $I$ is the collection of all Lebesgue null subsets of $\bbr$ then $\scr{B}_I^+(\bbr)$ denotes the collection of all Borel sets of positive measure.

Let us notice that similarly to the context of ideals we can define a base of $\scr{B}^+_I(X).$ In this case we say that $\scr{F}\subseteq \scr{B}^+_I(X)$ is a base if 
$$
(\forall A\in\scr{B}^+_I(X))(\exists B\in\scr{F})( B\subseteq A).
$$

Let us define
$$
cof(\scr{B}^+_I(X))=min\{ |\scr{A}|:\; \scr{A}\subset \scr{B}^+_I(X) \land \scr{A}\text{ is a base of } \scr{B}^+_I(X)\}.
$$

Any base of $\scr{B}^+_I(X)$ consists of Borel sets, so it is always of size smaller or equal to continuum. It shows that $cof(\scr{B}^+_I(X)) \le \mathfrak{c}$.

Let us remark that $cof(\cn) = cof(\scr{B}^+_\cn)(X)$ and $cof(\cm) = cof(\scr{B}^+_\cm)(X)$, see e.g. \cite{CKP} \cite{z1}.

\begin{definition}
	Let $(X,I)$ be a Polish ideal space. Assume that $C\subseteq  X.$
	We say that $C$ is $I$-nonmeasurable iff $C$ does not belong to the $\sigma$-field genereated by Borel subsets of $X$ and the $\sigma$-ideal $I$. 
\end{definition}

\begin{definition}
Let $(X,I)$ be a Polish ideal space. Assume that $C\subseteq D\subseteq X.$
We say that $C$ is completely $I$-nonmeasurable in $D$ iff
$$(\forall B\in\scr{B}^+_I(X))\left(B\cap D\notin I\then (B\cap C\notin I\wedge B\cap(D\setminus C)\notin I)\right).$$ 
\end{definition}

Let us remark that this notion was studied e.g. in papers \cite{rz1}, \cite{rz2}, \cite{cmrrz}, \cite{rsz}, \cite{z}.

\begin{definition} 
We say that a pair $(G,\cdot)$ is a Polish group if $(G,\cdot)$ is a topological group (multiplication and inverse are continuous) and $G$ with its topology is a Polish space. Neutral element of the group $(G,\cdot)$ will be denoted by $e$. If $X$ is any set and $(G,\cdot)$ is a group then a mapping $F: G\times X\to X$ is called an  action of $G$ on $X$ if the following conditions are fulfilled:
\begin{enumerate}
 \item $(\forall x\in X) F(e,x)=x$,
 \item $(\forall x\in X)(\forall g,h\in G)\;\; F(g,F(h,x))=F(gh,x)$.
\end{enumerate}
For the reader's convenience $F(g,x)$ is denoted by $gx$ and $(G,\cdot)$ by $G$.
\end{definition}

In the same fashion as in Definition \ref{pis} we define the notion of a Polish ideal group.
\begin{definition} We say that a triple $(G,\cdot, \cj)$ is a Polish ideal group if $(G,J)$ is a Polish ideal space and $(G,\cdot)$ is a Polish group.
\end{definition}

 \begin{definition} Let $G$ be any group and $X$  any set. Let $A\subset X$ be any subset of the $X.$ The set
$$
GA=\{ gx\in X:\;\; (g,x)\in G\times A\}
$$
is called the orbit of the set $A$ by the group $G$. Whenever $A=\{ x\}$ is a singleton, we will write $G{x}$ instead of $G{\{ x\}}$ for convenience.
\end{definition}

Let us notice that the family $\{ Gx: x\in X\}$ forms a partition of $X$. Each set of the form $Gx$ can be viewed as an abstract class with respect to the following equivalence relation:
$$
 (x\sim y)\iff (\exists g\in G)\; (y=gx)
$$

Let $A\subset G$ be any subset of the group $G.$ Then $\< A\>_G$ denotes the subgroup of $G$ generated by the set $A$. Let us observe that for any nonempty subset $A\subset G$ we have
$$
\< A \>_G= \bigg\{ \prod_{i\in n} f(i)^{g(i)}\in G:\;\; n\in\omega\land f\in A^n\land g\in \bbz^n\bigg\}.
$$
Here $h^0=e\in G$ whenever $h\in G$.

\section{Results }
The main inspiration of our paper is the question posed by Y. Kuznetsova \cite{Kuz}. She asked if for every null subset $A\subseteq \bbr$ there exists a subset $S\subseteq \bbr$ such that $A+S$ is nonmeasurable.

Z. Kostana in \cite{Kostana} proved that  for every meager subset $A\subseteq \bbr$ there exists a subset $S\subseteq \bbr$ such that $A+S$ does not have the Baire property.

We will consider similar problems for the wider class of Polish groups, not only $(\bbr,+).$ In most obtained results we used set-theoretical assumptions.

First, let us consider the situation when the set of orbits of singletons is quite large.
\begin{theorem} Let $(X,I)$ be a Polish ideal space and $(G,\cdot)$ an uncountable Polish group acting on $X.$ Assume that the following condition is fulfilled:
$$
(\forall B\in \scr{B}_I^+(X))\; cof(\scr{B}_I^+(X))\le |\{ Gb:\; b\in B\}|.
$$
Then there exists a  subset $A\subset X$ such that $A$ and $GA$ are completely $I$-nonmeasurable subsets of $X$. Moreover if $(G,J)$ is a Polish ideal space and there exist Borel bases $\scr{B}_G\subset \scr{B}_J^+(G)$ and $\scr{B}_X\subset \scr{B}_I^+(X)$ such that for every $B\in \scr{B}_I^+(X)$
$$
|\scr{B}_G|=|\scr{B}_X|\le |\{ Gb:\; b\in B\}|,
$$
then there exists a completely $J$-nonmeasurable subgroup $H$ in $G$ such that $A$ and $HA$ are completely $\ci$-nonmeasurable in $X.$
\end{theorem}
\begin{proof} We prove first the second part of the theorem, assuming that (G,J) is a Polish ideal space and there exist bases as in the statement.
Let us enumerate bases $\mathcal{B}_G=\{ C_\alpha:\alpha<\lambda\}$ and $\mathcal{B}_X=\{ B_\alpha:\alpha<\lambda\}$ where $\lambda\le \min\{ |\{ Gb:\; b\in B\}|: B \in \scr{B}_I^+(X) \}$. 

Now lets assume that we are in $\alpha<\lambda$ step of the construction with the following transfinite sequence:
$$
\< (a_\xi,d_\xi,h_\xi,c_\xi)\in B_\xi\times B_\xi\times C_\xi\times  C_\xi:\;\; \xi<\alpha\>
$$
with the following conditions:
\begin{enumerate}
 \item the collection of orbits $\{ G{a_\xi}:\xi<\alpha\}\cup \{ G{d_\xi}:\xi< \alpha\}$ is pairwise disjoint,
 \item $\< h_\xi:\; \xi<\alpha\>_G\cap \{ c_\xi:\xi<\alpha\}=\emptyset$.
\end{enumerate}
By the assumption we can find 
$$
a_\alpha,d_\alpha \in B_\alpha\setminus \bigcup(\{ G{a_\xi}:\xi<\alpha\}\cup \{ G{d_\xi}:\xi< \alpha\})
$$
We pick next any $h_\alpha\in C_\alpha$ and choose $c_\alpha \in C_\alpha\setminus \langle h_\xi: \xi\le \alpha\rangle_G$. This choice is possible 
 because $|\< Z \>_G|\le \aleph_0\cdot |Z|$ for any set $Z\subset G$. Then $\alpha$ step is finished. 

Now let us take the following sets: $H=\< h_\alpha\in G:\alpha<\lambda\>_G$, $A=\{ a_\alpha\in X:\alpha<\lambda\}$ and $D=\{ d_\alpha\in X:\; \alpha<\lambda\}.$ Then by our transfinite construction $H$ is completely J-nonmeasurable subgroup of $G$, $A$ and $D$ are completely I-nonmeasurable subsets of the Polish space $X$. Moreover by the following inclusion
$$
A\subset HA\subset D^c
$$
we see that $HA$ is completely I-nonmeasurable subset of $X$.

Now note that by construction, GA is contained in $D^c$. If now G is not supposed an ideal space, we carry over the same induction for $a_\alpha, d_\alpha$ but do not choose $h_\alpha$ or $c_\alpha$. We still get a completely I-nonmeasurable A and $GA\subset D^c$, so that A and GA are completely I-nonmeasurable.
\end{proof}

By a simple modification of the proof of the previous Theorem (by double transfinite induction) we can derive a stronger result:

\begin{theorem}
\label{many_orbits} Let $(X,I)$ be a Polish ideal space and $(G,\cdot)$ be an uncountable group acting on $X$. Assume that the following condition is fulfilled:
$$
(\forall B\in \scr{B}_I^+(X))\; cof(\scr{B}_I^+(X))\le |\{ Gb:\; b\in B\}|.
$$
Then there exists a subgroup $H\le G$ and a pairwise disjoint family $\{A_\alpha:\alpha< cof(\scr{B}_I^+(X))\}\subset \scr{P}(X)$ such that: 
\begin{enumerate}
 \item $(\forall \alpha< cof(\scr{B}_I^+(X)))\; A_\alpha, H{A_\alpha}$ are completely $I$-nonmeasurable in $X$,
 \item $(\forall \alpha,\beta)\alpha<\beta< cof(\scr{B}_I^+(X)) \then H{A_\alpha}\cap H{A_\beta}=\emptyset$.
\end{enumerate}
Moreover if $(G,J)$ is a Polish ideal space and there exist Borel bases $\mathcal{B}_G\subset \scr{B}^+_J(G)$ and $\mathcal{B}_X\subset \scr{B}_I^+(X)$ with
$$
|\mathcal{B}_G|=|\mathcal{B}_X|\le |\{ Gb:\; b\in B\}|,
$$
then $H$ can by choosen completely $J$-nonmeasurable in the group $G$.
\end{theorem}


Now, let us now consider the situation when, in contrast to previous theorems, there exists only one orbit.

\begin{theorem}\label{gg} Let $(X,I)$ be a Polish ideal space. Assume that $(G,\cdot, J)$ forms a Polish ideal group acting on $X.$ 
If for some (every) $x\in X$ $Gx=X$ and 
$$
(\exists \lambda<2^\omega)(\forall x,y\in X)\; x\ne y\then |G_{x,y}|\le \lambda
$$
where $G_{x,y}=\{ g\in G:\; y=gx\}$ then there exists a subgroup $H\le G$ and a subset $A\subset  X$ such that $A$ and $HA$ are completely I-nonmeasurable sets in $X$ and $H$ is completely $J$-nonmeasurable in $G$.
\end{theorem}
\begin{proof} 
	As before, suppose first that $(G,J)$ is a Polish ideal space. Let us enumerate the bases of $\scr{B}_J^+(G)$ and $\scr{B}_I^+(X)$, 

$$
\mathcal{B}_G=\{ C_\alpha:\alpha<2^\omega\} \text{ and } \mathcal{B}_X=\{ B_\alpha:\alpha<2^\omega\}.
$$
If there is no topology supposed on G, we set $C_\alpha=G$ for every $\alpha<2^\omega.$
Let us suppose that we are in the $\alpha$-step of recursive construction:
$$
\< (a_\xi,d_\xi,h_\xi,c_\xi)\in B_\xi\times B_\xi\times C_\xi\times C_\xi:\xi<\alpha\>
$$
with the following conditions:
$$
{H_\alpha}{A_\alpha}\cap D_\alpha=\emptyset,\text{ and } H_\alpha\cap F_\alpha=\emptyset,
$$

where $H_\alpha=\< \{ h_\xi:\xi<\alpha\}\>_G$, $F_\alpha=\{ c_\xi:\xi<\alpha\}$, $A_\alpha=\{ a_\xi:\xi<\alpha\}$ and $D_\alpha=\{ d_\xi:\xi<\alpha\}$.

Now let us choose any element  $a_\alpha\in B_\alpha\setminus D_\alpha$. 
Set $W_\alpha$ to be the following set
$$
\{ h\in G: (\exists \xi<\alpha) (d_\xi \in \langle \{h\}\cup H_\alpha\rangle_G a_\alpha) \}
$$


For every $h\in G$, the subgroup $H_{\alpha,h} = \langle \{h\}\cup H_\alpha\rangle_G$  has the cardinality less or equal to $\aleph_0|H_\alpha|$. By definition, for every $h\in W_\alpha$ there exists $\xi<\alpha$ such that $ H_{\alpha,h}\cap G_{a_\alpha, d_\xi} \ne \emptyset$, so that the cardinality of $W_\alpha$ is less or equal to $\lambda \aleph_0 |H_\alpha|$. Let us choose any $h_\alpha\in C_\alpha\setminus W_\alpha$. It is possible because $C_\alpha$ contains some uncountable perfect set. 
Now choose any element $c_\alpha\in C_\alpha\setminus \< \{ h_\alpha\}\cup H_\alpha\>_G$ by the same argument. Finally we can choose any element $d_\alpha\in B_\alpha\setminus \< \{ h_\alpha\}\cup H_\alpha\> ( \{ a_\alpha\}\cup A_\alpha)$ what is possible because we have the following cardinal inequality:
$$
|\< \{ h_\alpha\}\cup H_\alpha\> (\{ a_\alpha\}\cup A_\alpha) |\le |\< \{ h_\alpha\}\cup H_\alpha\>| \cdot |\{ a_\alpha\}\cup A_\alpha|\le |\alpha|\aleph_0<2^\omega.
$$
Then our construction is done by transfinite induction Theorem. Let us observe that the following sets
$$
A=\{ a_\alpha:\alpha<2^\omega\}, D=\{ d_\alpha:\alpha<2^\omega\}, H=\< h_\alpha:\alpha<2^\omega\>_G=\bigcup_{\alpha<2^\omega} H_\alpha
$$
fulfill assertion of this Theorem because $A\subset HA\subset D^c$. So, the proof is finished.
\end{proof}

With a little more effort we can prove the analogous theorem in the case when the set of all orbits is small and those orbits are $I$-measurable subsets of $X$ (which means that belong to the $\sigma$-algebra $\scr{B}[I]$ generated by the sets which are Borel or they belong to the ideal $I$). Namely, we have the following result. 

\begin{theorem}\label{few_orbits}
 Let $(X,I)$ be a Polish ideal space.
Assume that $(G,\cdot,J)$ forms a Polish ideal  group acting on $X.$ 
Assume that $X=\bigcup \{   G{x_n}:n\in\omega\}$ is the union of  countably many $I$-positive and $I$-measurable orbits. Suppose that
$$
(\exists \lambda<2^\omega)(\forall x,y\in X)\; x\ne y\then |G_{x,y}|\le \lambda,
$$
where $G_{x,y}=\{ g\in G:\; y=gx\}.$ Then there exists a subgroup $H\le G$ and countably many families  $\scr{A}_n=\{ A_\alpha^n:\alpha< 2^\omega\}$, $n\in\omega ,$ of continuum many pairwise disjoint subsets of $X$ satisfying  the following properties
$$
(\forall n\in\omega)(\forall \alpha<2^\omega) A_\alpha^n, H{A_\alpha^n}\text{ are completely }I-\text{nonmeasurable in }G{x_n},
$$
$H$ is completely $J$-nonmeasurable in $G$.
\end{theorem}

\begin{theorem}\label{hghg}
Let $(G,\cdot,\cj)$ be a Polish ideal group which acts on a Polish ideal space $(X,\ci)$. Let us assume that 
\begin{enumerate}
    \item $cov_h(\cj)=cof(\scr{B}_J^+(G))=cof(\scr{B}_I^+(X))$, 
    \item there exists $G'\subseteq G$ with $G\setminus G'\in J$ such that 
    	for any $n\in\omega$, $s\in \bbz^n$  and for every $g\in G'$, $a\in G'^n$ the following condition holds:
    $$
    S_{a,s,g}  = \{ h\in G:\; \prod_{i\in n} a_i\cdot h^{s_i} = g\} \in\cj,
    $$
    \item there exists $X' \subseteq X$ such that $X\setminus X' \in \ci$ such that  
    for any $n\in\omega$, $s\in \bbz^n$, for every $a\in G'^n$ and every $x,y\in X'$ the following condition holds:
    $$
    T_{a,s.x,y} = \{ h\in G:\; \big(\prod_{i\in n} a_i\cdot h^{s_i} \big) x = y\} \in\cj.
    $$
    
\end{enumerate}
Then there is a completely $\cj$-nonmeasurable subgroup $H \le G$  and a completely $\ci$-nonmeasurable subset $A\subseteq X$ such that $HA$ is completely $\ci$-nonmeasurable in the space $X$.
\end{theorem}
\begin{proof} Denote $\kappa = cof(\scr{B}_J^+(G))=cof(\scr{B}_I^+(X))$.
	Let us enumerate the bases of $\scr{B}_J^+(G)$ and $\scr{B}_I^+(X)$, 
	
	$$
	\mathcal{B}_G=\{ P_\xi:\xi< \kappa\} \text{ and } \mathcal{B}_X=\{ B_\xi:\xi< \kappa\}.
	$$
	We assume that $\mathcal{B}_G$ consists of subsets of $G'$ and $\mathcal{B}_X$ consists of subsets of $X'.$
	
	By transfinite recursion we will define a sequence of the length $\kappa$:
$$
\< (H_\xi,A_\xi,g_\xi,d_\xi):\ \xi<\kappa\>
$$
such that for any $\xi,\eta<\kappa$, we have
\begin{enumerate}
	\item $g_\xi \in P_\xi$ and $d_\xi \in B_\xi,$ 
    \item $H_\xi\le G$ and $H_\xi\cap P_\xi\ne\emptyset$,
    \item $A_\xi\subseteq X$ and $A_\xi\cap B_\xi\ne\emptyset$,
    \item if $\xi\le \eta$ then $H_\xi\subseteq H_\eta$ and $A_\xi\subseteq A_\eta$,
    \item whenever $\xi\le \eta$ then $g_\xi\notin H_\eta$ and $d_\xi\notin H_{\eta}{A_{\eta}}$.
\end{enumerate}

Let us consider $\alpha$-th step of our transfinite induction. Set $H'_\alpha=\bigcup_{\xi<\alpha} H_\xi$ and $A'_\alpha=\bigcup_{\xi<\alpha} A_\xi.$  
Define $F_\alpha = \{ g_\xi: \xi<\alpha\}$ and $D_\alpha = \{ d_\xi: \xi<\alpha\}$.

Observe that $H'_\alpha D_\alpha$ has cardinality less than $\kappa.$  So we can find $a_\alpha \in B_\alpha \setminus H'_\alpha D_\alpha$.
Define $A_\alpha = A'_\alpha\cup \{ a_\alpha \}$. 

Define 
$$
Z_\alpha = \bigcup_{n\in\omega} \bigcup_{s\in\bbz^n} \bigcup_{a\in F_\alpha^n} \bigcup_{g\in F_\alpha} S_{a,s,g}
$$
and
$$
X_\alpha = \bigcup_{n\in\omega} \bigcup_{s\in\bbz^n} \bigcup_{x\in A_\alpha} \bigcup_{y\in D_\alpha} T_{a,s,x,y}.
$$
Notice that $X_\alpha\cup Z_\alpha$ is a union of less than $cov_h(\cj)$ sets from $\cj.$ Thus there exists $h_\alpha \in P_\alpha\setminus (X_\alpha \cup Z_\alpha).$

Define $H_\alpha = \< H'_\alpha \cup \{ h_\alpha\} \>.$ By the definition of $X_\alpha$ we have $A_\alpha\cap H_\alpha D_\alpha = \emptyset.$ So $H_\alpha A_\alpha \cap D_\alpha = \emptyset.$

To finish the $\alpha$-step of  construction we can choose $d_\alpha \in B_\alpha \setminus H_\alpha A_\alpha$ and $g_\alpha\in  P_\alpha \setminus H_\alpha.$

The sets $H=\bigcup_{\xi<\kappa} H_\xi$ and $A=\bigcup_{\xi<\kappa} A_\xi$ gives the required assertion.
\end{proof}

The conditions (2) and (3) from the previous Theorem looks artificially. However, we can consider a natural situation, where  $J$ is the ideal of meager sets $\cm.$
It leads to the following proposition.

\begin{proposition}\label{meager}
Let $(G,\cdot)$ be an uncountable Polish group. Fix $n\in\omega$, $s\in \bbz^n.$ Then there exists comeager $G'\subseteq G$ such that  for every $g\in G'$, $a\in G'^n$ the following set
    $$
    \{ h\in G:\; \prod_{i\in n} a_i\cdot h^{s_i} = g\} 
    $$
is meager.    
\end{proposition}
\begin{proof} Set $M=\{(h,g,a_0,\ldots,a_{n-1})\in G^{n+2}:\; \prod_{i\in n} a_i\cdot h^{s_i} = g\}.$ $M$ is a closed subset of $G^{n+2}.$ Moreover, its interior is empty. So, $M$ is a closed nowhere dense set. By Kuratowski-Ulam theorem  comeager many first sections of $M$ are meager. Using Kuratowski Ulam theorem, the set $\{ h\in G:\; \prod_{i\in n} a_i\cdot h^{s_i} = g\}$ is meager for $(g,a_0,\ldots,a_{n-1})\in C$, where $C\subseteq G^{n+1}$ is comeager. Using Kuratowski-Ulam theorem $n$-many times we can find $G'$ -- a comeager subset of $G$ such that $G'^{n+1}\subseteq C.$ Then $G'$ is the required set.
\end{proof}

Analogously we can prove the next proposition.
\begin{proposition}\label{meager_1}
	Let $(G,\cdot)$ be an uncountable Polish group acting on an uncountable Polish space $X$. Fix $n\in\omega$, $s\in \bbz^n.$ Then there exists comeager $G'\subseteq G$ and comeager $X'\subseteq X$  such that  for every $a\in G'^n$ and every $x,y\in X'$ the following set
	$$
	T_{a,s.x,y} = \{ h\in G:\; \big(\prod_{i\in n} a_i\cdot h^{s_i} \big) x = y\}
	$$
	is meager.
\end{proposition}

\begin{theorem} Let $(X,\ci)$ be a Polish ideal space and  $non(\ci)<cov_h(\ci)$. Let $(G,\cdot)$ be a group which acts on $X$. Assume that $H\le G$ and $A\in\ci$ are such that $HA$ contains a Borel set $B\notin\ci.$ Then there is a subgroup $H'\le H$ such that $H'A$ is  $\ci$-nonmeasurable.
\end{theorem}
\begin{proof} Let $B\in \scr{B}^+_I(X)$ be an $\ci$-positive Borel set such that $B\subseteq HA.$ Let us pick a set $T\subset B$ such that $T\notin \ci$ and  $|T| = non(\ci)$. Next, let $F:T\to H$ be such that  $t\in F(t)A$ holds for any $t\in T$. Let $H'=\< F[T]\>$ be the subgroup of $H$ generated by $F[T].$ Then 
$$
|H'|=|F[T]|\le|T|=non(\ci)<cov_h(\ci).
$$
We have that $T\subseteq F[T]A\subseteq H'A.$ So, $H'A\notin\ci$. Notice that $H'A=\bigcup\cf,$ where the  family $\cf=\{ hA:\; h\in H'\}\subseteq\ci$ of sets from the $\sigma$-ideal $\ci$ has size less than $cov_h(I).$ It shows that no $\ci$-positive Borel set can be covered by the family $\cf$.
\end{proof}

\section{Applications}
In investigations of topological spaces $X$ the crucial role is played by the space of all homeomorphisms $\ch(X)$ on the space $X$. $\ch(X)$ is endowed with so-called compact-open topology which is generated by the following subbase 
$$
\{ V(K,U):\;K\subseteq X\text{ is compact and }U\subseteq X\text{ is open in }X\},
$$
where
$$
V(K,U)=\{ f\in\ch:\; f[K]\subseteq U\}).
$$
In the case when $X$ is a compact metric space, $\ch(X)$ is also metrizable. A metric on $\ch(X)$ can be defined for every $f,\ g\in\ch(X)$ by the following formula
$$
 d(f,g)=\sup_{x\in X}\{ d(f(x),g(x))\} + \sup_{y\in X} d(f^{-1}(y),g^{-1}(y)).
$$
When $X$ is a compact Polish space then $\ch(X)$ is also a Polish space, see e.g. \cite{Chapman}.

Let us now consider a natural context, when a Polish group G acts on itself by left shifts. As a consequence of Theorem \ref{gg} we get the following result.

\begin{cor}
Let $(G,\cdot,J)$ be a Polish ideal group. Then there exist $H<G$ and $A\subseteq G$ such that $H,\, A,\, HA$ are completely $J$-non\-mea\-su\-ra\-ble. 
\end{cor}
\begin{proof}
It is enough to check that the assumptions of Theorem \ref{gg} are fulfilled.

Naturally $Ge=\{ge:\ g\in G\}=G.$

Notice that if $x\neq y$, $x,y\in G$ then
$$
G_{x,y}=\{g\in G:\ y=gx\}=\{yx^{-1}\},
$$
so, it has size 1.
\end{proof}

Now, let us consider a situation when a group of homeomorphisms $\scr{H}(X)$ acts in the natural way on a compact Polish space $X.$

\begin{cor}
Assume that $cov(\cm)=cof(\cm).$ Let $X$ be a compact Polish space without isolated points. Then there exist a completely $\cm$-nonmeasurable subgroup $H\le \scr{H}(X)$ and a completely $\cm$-non\-mea\-su\-ra\-ble subset $A\subseteq X$ such that $HA$ is completely $\cm$-nonmeasurable.
\end{cor}
\begin{proof}
We will use Theorem \ref{hghg}. The assumptions are fulfilled, what follows from Proposition  \ref{meager} and Proposition \ref{meager_1}.
\end{proof}


On the other hand, the following example is a simple corollary of Theorem \ref{many_orbits}.
\begin{cor} Let $G$ be the subgroup of the group of all isometries on the Cantor space $2^\omega$ defined as follows
$$
G=\{ T_X:\; X\in \scr{P}(\{ n\in\omega:\; n\equiv 0\mod 2\})\},
$$
where for any $x\in 2^\omega$ and $n\in\omega$
$$
T_X(x)(n)=\begin{cases}
x(n)    &\text{when }n\notin X\\
1-x(n)  &\text{when }n\in X.
\end{cases}
$$
Then there is a subgroup $H$ of $G$ and uncountable many pairwise disjoint subsets $\{ A_\alpha\subset 2^\omega:\; \alpha< cof(\cm)\}$ such that $HA_\alpha$ are completely $\cm$-nonmeasurable in the Cantor space $2^\omega$ for any $\alpha < cof(\cm)$. Moreover,  $\{ HA_\alpha:\alpha< cof(\cm)\}$ forms a pairwise disjoint family of subsets of the Cantor space.
\end{cor}


\end{document}